\documentclass[12pt]{amsart}
\usepackage{amsmath}
\usepackage{amsthm}
\usepackage{amssymb}
\usepackage{amscd}
\usepackage{mathrsfs}
\usepackage{txfonts}
\usepackage{color}
\newtheorem{theorem}{Theorem}[section]
\newtheorem{proposition}[theorem]{Proposition}
\newtheorem{lemma}[theorem]{Lemma}
\newtheorem{corollary}[theorem]{Corollary}

\theoremstyle{remark}
\newtheorem{remark}[theorem]{Remark}

\theoremstyle{definition}
\newtheorem{definition}[theorem]{Definition}

\newcommand{\Z}{\mathbb Z}
\newcommand{\Q}{\mathbb Q}
\newcommand{\R}{\mathbb R}
\newcommand{\C}{\mathbb C}
\newcommand{\SU}{\mathit{SU}}

\newcommand{\SL}{\mathit{SL}}
\newcommand{\PSL}{\mathit{PSL}}
\newcommand{\tr}{\mathop{\mathrm{tr}}\nolimits}
\newcommand{\sh}{\mathop{\mathrm{sh}}\nolimits}
\begin{document}

\title{$\SL(2;\R)$-representations of a Brieskorn homology 3-sphere}

\author{Teruaki Kitano and Yoshikazu Yamaguchi}

\begin{abstract}
We classify $\SL(2;\C)$-representations of a Brieskorn homology 3-sphere.
We show any irreducible representation into $\SL(2;\C)$ 
is conjugate to that into 
either $\SU(2)$ or $\SL(2;\R)$.
We also give a construction of  
$\SL(2;\R)$-representations for a Brieskorn homology 3-sphere
from $\PSL(2;\R)$-representations of the base orbifold fundamental group.
\end{abstract}



\address{Department of Information Systems Science, 
Faculty of Science and Engineering, 
Soka University, 
Tangi-cho 1-236, 
Hachioji, Tokyo 192-8577, Japan}

\email{kitano@soka.ac.jp}

\address{Department of Mathematics, 
Akita University, 
1-1 Tegata-Gakuenmachi, Akita, 010-8502, Japan}

\email{shouji@math.akita-u.ac.jp}
\maketitle

\section{Introduction}\label{section:1}
In~\cite{Yamaguchi}, the second author studied the asymptotic behavior 
of the Reidemeister torsion for Seifert fibered spaces with
$\SL(2;\C)$-representations which are not realized by $\SU(2)$-ones.
We are interested in finding out what is the difference 
between irreducible $\SL(2;\C)$-representations and $\SU(2)$-ones
in the cases of Brieskorn homology 3-spheres.
We will see that irreducible $\SL(2;\C)$-representations consists of
$\SU(2)$-ones and $\SL(2;\R)$-ones, up to conjugate. 
The latter ingredients are constructed by homomorphisms from 
the fundamental group of the base orbifold into $\PSL(2;\R)$. 
We give a concrete correspondence between
$\SL(2;\R)$-representations of a Brieskorn homology 3-sphere
and 
$\PSL(2;\R)$-representations 
of the fundamental group of the base orbifold.
This correspondence is derived by covering maps 
in a sequence of Seifert fibered spaces.

Let $\Sigma(a_1,a_2,a_3)$ be a Brieskorn homology 3-sphere 
which is defined by
\[
\Sigma(a_1,a_2,a_3)
=\{(z_1,z_2,z_3)\in\C^3\ |\ z_1^{a_1}+z_2^{a_2}+z_3^{a_3}=0,  |z_1|^2+|z_2|^2+|z_3|^2=1
\}.
\]
Here $a_1,a_2,a_3$ are pairwise coprime positive integers. 
We may assume $a_2$, $a_3$ are odd without the loss of generality.

Our main result in this paper is the following. 
\begin{theorem}
  Let $\rho$ be an irreducible representation
  of $\pi_1 \Sigma(a_1, a_2, a_3)$ into $\SL(2;\C)$.
  If $\rho$ is not conjugate to an $\SU(2)$-representation, 
  then $\rho$ is conjugate to an $\SL(2;\R)$-representation which is induced by  
  a homomorphism from the fundamental group
  of the base orbifold into $\PSL(2;\R)$.
\end{theorem}

It is well known that $\Sigma(a_1,a_2,a_3)$ admits a structure of Seifert fibered space  
\[
\Sigma(a_1,a_2,a_3)\rightarrow S^2(a_1,a_2,a_3)
\]
such that the base orbifold is $S^2(a_1,a_2,a_3)$, which is topologically $S^2$ with three cone points. 
We write $\Gamma(a_1,a_2,a_3)$ to the orbifold fundamental group of $S^2(a_1,a_2,a_3)$ 
with a presentation 
\[
\Gamma(a_1,a_2,a_3)=
\langle
x,y,z\ |\ x^{a_1}=y^{a_2}=z^{a_3}=xyz=1
\rangle.
\]

In this article, 
we consider irreducible representations of $\pi_1\Sigma(a_1,a_2.a_3)$ into Lie groups $\SL(2;\C),\SU(2),\SL(2;\R)$. 
For a Lie group $G$,  
we denote the space of conjugacy classes of irreducible representations 
of $\pi_1\Sigma(a_1,a_2,a_3)$ into $G$ by $\mathcal{R}_G$. 
\begin{remark}
  We consider conjugacy classes of irreducible representations
  under conjugation by $\SL(2;\C)$.
\end{remark}

There are many studies on $\mathcal{R}_{\SU(2)}$, and $\mathcal{R}_{\SL(2;\C)}$ 
which are related with $\SU(2)$, or $\SL(2;\C)$-Casson invariant of a homology 3-sphere. 

In~\cite{Fintushel-Stern} 
Fintushel and Stern proved that $\mathcal{R}_{\SU(2)}$ is a finite set. 
In particular the number is expressed as 
\[
|\mathcal{R}_{\SU(2)}|
=\frac{(a_1-1)(a_2-1)(a_3-1)}{4}-|X_0|
\]
where the finite set $X_0$ is defined by
\[
X_0=\{(k,l,m)\ |\ 0<k<a_1,0<l<a_2,0<m<a_3,\frac{k}{a_1}+\frac{l}{a_2}+\frac{m}{a_3}<1\}.
\]
They also proved that 
\[
|\mathcal{R}_{\SU(2)}|= 2 |\lambda(\Sigma(a_1,a_2,a_3))|
\]
where $\lambda(\Sigma(a_1,a_2,a_3))$ is the Casson invariant of $\Sigma(a_1,a_2,a_3)$. 

\begin{remark}
In fact they studied the space of conjugacy classes of $\SU(2)$-representations 
under conjugation by $\SU(2)$ itself. 
However it is equal to $\mathcal{R}_{\SU(2)}$ for a Brieskorn homology 3-sphere.
\end{remark}

Curtis~\cite{Curtis} defined 
the $\SL(2;\C)$-Casson invariant for homology 3-spheres 
by counting conjugacy classes of irreducible $\SL(2;\C)$-representations. 
Further Boden and Curtis~\cite{Boden-Curtis} proved the equality that 
\[
\begin{split}
  \frac{(a_1-1)(a_2-1)(a_3-1)}{4}
  &=|\mathcal{R}_{\SL(2;\C)}|\\
  &=\lambda_{\SL(2;\C)}(\Sigma(a_1,a_2,a_3))
\end{split}
\]
by looking into representations of $\Gamma(2a_1,a_2,a_3)$ in $\SL(2;\C)$.

Hence Fintushel-Stern's formula can be expressed as 
\[
\lambda_{\SL(2;\C)}(\Sigma(a_1,a_2,a_3))-2|\lambda(\Sigma(a_1,a_2,a_3))|
=|X_0|.
\]

On the other hand, 
Jankins and Neumann~\cite{Jankins-Neumann85} 
studied the space of 
homomorphisms from $\Gamma(a_1,a_2,a_3)$
into $\PSL(2;\R)$. 
It is not a finite set, but connected components 
can be classified by euler classes for central extensions 
of $\Gamma(a_1,a_2,a_3)$ by $\Z$, 
which belong to $H^2(\Gamma(a_1,a_2,a_3);\Z)$.  
The number of connected components coincides with $2|X_0|$.

\begin{corollary}
  We have that 	 
  \[\lambda_{\SL(2;\C)}(\Sigma(a_1,a_2,a_3))
  -2|\lambda(\Sigma(a_1,a_2,a_3))|
  =|\mathcal{R}_{\SL(2;R)}|.\]
\end{corollary}

\section{Preliminaries}\label{section:preliminary}
\subsection{Seifert invariants} 
We say that 
a Seifert fibered space with three singular fibers over the sphere
has a Seifert invariant
\[
\{0;(1,b), (a_1,b_1), (a_2,b_2), (a_3,b_3)\}
\] 
where $b$ and $b_i$ are integers. 
Here $0$ is the genus of $S^2(a_1,a_2,a_3)$. 
We write 
$M(0;(1,b), (a_1,b_1), (a_2,b_2), (a_3,b_3))$
to the Seifert fibered space by the above invariant. 
Please see~\cite{Jankins-Neumann83, Orlik} 
for general theory of Seifert fibered spaces.

By using these invariants, 
the fundamental group $\pi_1\Sigma(a_1,a_2,a_3)$ has the following presentation:
\[
\langle
x,y,z,h\ |\ h\text{ central}, x^{a_1}=h^{-b_1},y^{a_2}=h^{-b_2},z^{a_3}=h^{-b_3},xyz=h^{b}
\rangle.
\]

We use the following convention of the euler number for 
a Seifert fibration. 

\begin{definition}
  For
  $M=M(0;(1, b),(a_1,b_1),(a_2,b_2),(a_3,b_3))$,
  we will denote the euler number by $e(M)$ defined as
  \[
  e(M)
  =-\left(
  b+\frac{b_1}{a_1}+\frac{b_2}{a_2}+\frac{b_3}{a_3} 
  \right)\in\Q .
  \]
\end{definition}

\begin{proposition}
  For any $M=M(0;(1,b);(a_1,b_1),(a_2,b_2),(a_3,b_3))$, 
  one has
  \[
  H_1(M;\Z)
  \simeq\Z/a|e(M)|
  \]
  where $a=a_1 a_2 a_3$. 
\end{proposition}
For the proof, see~\cite{Jankins-Neumann83}.

\begin{corollary}
  If $a |e(M)| =1$,
  then $M$ is a homology 3-sphere. 
  In particular $M\cong\Sigma(a_1,a_2,a_3)$. 
\end{corollary}

Here and subsequently, 
we take $b,b_1,b_2,b_3\in\Z$ as 
\[
  a\left(
  b+\frac{b_1}{a_1}+\frac{b_2}{a_2}+\frac{b_3}{a_3} 
  \right)
  = 1.
\]

Since $h$ is corresponding to the regular fiber of  $\Sigma(a_1,a_2,a_3)$, 
one has a short exact sequence 
\[
1\rightarrow
\langle h\rangle
\rightarrow
\pi_1\Sigma(a_1,a_2,a_3)
\rightarrow
\Gamma(a_1,a_2,a_3)=\pi_1\Sigma(a_1,a_2,a_3)/\langle h\rangle
\rightarrow 1.
\]
Further this is a central extension of $\Gamma(a_1,a_2,a_3)$ by $\langle h\rangle\cong\Z$. 

If one choose integers $\beta$ and $\beta_i$ such that 
\[
a\left|
\beta+\frac{\beta_1}{a_1}+\frac{\beta_2}{a_2}+\frac{\beta_3}{a_3} 
\right|
\neq 1,
\]
then 
the corresponding Seifert fibered space
\[\overline{M} = M(0;(1, \beta), (a_1, \beta_1), (a_2, \beta_2), (a_3, \beta_3))\]
is not a homology 3-sphere. 
We use the symbols $\overline{M}$, $\beta$ and $\beta_i$ for a Seifert rational homology 3-sphere.

\subsection{$\SL(2;\R)$-representation}\label{section:4}
The projective group 
$\PSL(2;\R)$ is the quotient $\SL(2;\R) / \{\pm I\}$ 
where $I$ is the identity matrix.
$\PSL(2;\R)$ acts on $\R P^1\cong S^1$ as protectively automorphisms. 
Let $\widetilde{\PSL}(2;\R)$ be the universal cover of $\PSL(2;\R)$. 
$\widetilde{\PSL}(2;\R)$ also acts on $\R^1$ which is the universal cover of $\R P^1$. 
Then $\widetilde{\PSL}(2;\R)$ can be considered as a subgroup 
in the group of homeomorphisms of $\R$. 
For any $r\in\R$, 
we define 
\[
\sh(r):\R\ni x\mapsto x+r\in\R.
\]
The center of $\widetilde{\PSL}(2;\R)$ is given by 
$\{\sh(n):\R\rightarrow\R\ |\ n\in\Z\} = \langle \sh(1) \rangle \cong \Z$. 

\begin{lemma}
  This short exact sequence 
  \[
  0\rightarrow\Z\rightarrow
  \widetilde{\PSL}(2;\R)\rightarrow\PSL(2;\R)
  \rightarrow 0\]
  is a central extension of $\PSL(2;\R)$ by $\Z=\langle \sh(1) \rangle$.
\end{lemma}

Recall that $\Gamma(a_1,a_2,a_3)$ is the quotient group of $\pi_1\Sigma(a_1,a_2,a_3)$ 
by the center $\langle h \rangle$ with the following presentation:
\[
\Gamma(a_1,a_2,a_3)
=\langle x,y,x \ |\ x^{a_1}=y^{a_2}=z^{a_3}=xyz=1\rangle.
\]

Given a homomorphism 
$f:\Gamma(a_1,a_2,a_3)\rightarrow \PSL(2;\R)$, 
we take the pull--back of the central extension 
\[
0\rightarrow\Z\rightarrow
\widetilde{\PSL}(2;\R)\rightarrow\PSL(2;\R)
\rightarrow 0.\]
This gives a commutative diagram 
\[
\begin{CD}
0@>>>\Z @>>>\widetilde{\Gamma}(a_1,a_2,a_3)@>>>\Gamma(a_1,a_2,a_3)@>>> 0\\
@. @| @VVV @VVfV \\ 
0@>>>\Z @>>>\widetilde{\PSL}(2;\R)@>>>\PSL(2;\R)@>>> 0.
\end{CD}
\]

Here $\widetilde{\Gamma}(a_1,a_2,a_3)$ has the following presentation:
\[
\langle
x,y,z,h\ |\ h\text{ central}, 
x^{a_1}=h^{-\beta_1},
y^{a_2}=h^{-\beta_2},
z^{a_3}=h^{-\beta_3},
xyz=h^{\beta}
\rangle
\]

The classification of the central extension of $\Gamma(a_1,a_2,a_3)$ is given by 
the euler class $eu(f)$ in 
\[
H^2(\Gamma(a_1,a_2,a_3);\Z) \cong
\langle x_0,x_1,x_2,x_3\ |\ a_1x_1=x_0, a_2x_2=x_0, a_3x_3=x_0 \rangle_{\Z}
\]
such that 
$eu(f)=\beta x_0+\beta_1 x_1+\beta_2 x_2+\beta x_3$.
By the relations, 
one can normalizes any element $\gamma \in H^2(\Gamma(a_1,a_2,a_3);\Z)$
as 
\[
\gamma=\beta x_0+\beta_1 x_1+\beta_2 x_2+\beta_3 x_3
\]
such that $0<\beta_1<a_1, 0<\beta_2<a_2$ and $ 0<\beta_3<a_3$. 

Applying~\cite[Theorem~1]{Jankins-Neumann85} to our case that
$a_1,a_2,a_3$ are pairwise coprime, we have the following theorem.
\begin{theorem}[Jankins-Neumann~\cite{Jankins-Neumann85}]
  \label{thm:JankinsNeumann}
  For any $\gamma=\beta x_0+\beta_1 x_1+\beta_2 x_2+\beta x_3$, 
  it equals $eu(f)$ for some $f:\Gamma(a_1,a_2,a_3)\rightarrow\PSL(2;\R)$ 
  if and only if the following is true;
\item[(a)]
  $\beta=-1, 
  \displaystyle\frac{\beta_1}{a_1}+\frac{\beta_2}{a_2}+\frac{\beta_3}{a_3} < 1$,

  either
\item[(b)]
  $\beta=-2, 
  \displaystyle\frac{\beta_1}{a_1}+\frac{\beta_2}{a_2}+\frac{\beta_3}{a_3} > 2 $.
\end{theorem}

Now assume that such an $f:\Gamma(a_1,a_2,a_3)\rightarrow \PSL(2;\R)$ exists. 
Then there exists 
a central extension 
\[
0\rightarrow\langle h \rangle
\rightarrow 
\widetilde{\Gamma}(a_1,a_2,a_3)
\rightarrow \Gamma(a_1,a_2,a_3)\rightarrow 1
\] 
defined by the pull--back of 
\[
0\rightarrow\Z
\rightarrow \widetilde{\PSL}(2;\R)
\rightarrow \PSL(2;\R)
\rightarrow 1
\]
such that 
$\widetilde{\Gamma}(a_1,a_2,a_3) \ni h\mapsto \sh(1)\in\widetilde{\PSL}(2;\R)$. 

The following holds
from the proof of Theorem~\ref{thm:JankinsNeumann} in~\cite{Jankins-Neumann85}
\begin{proposition}
  $\widetilde{\Gamma}(a_1,a_2,a_3)$ is the fundamental group of 
  the Seifert fibered space with Seifert invariant 
  $\{0;(1,\beta),(a_1,\beta_1),(a_2,\beta_2),(a_3,\beta_3)\}$. 
\end{proposition}

Hence by composing $\widetilde{\Gamma}(a_1,a_2,a_3)\rightarrow \widetilde{\PSL}(2;\R)$ 
with the projection from 
$\widetilde{\PSL}(2;\R)$ onto $\SL(2;\R)$, 
one has an $\SL(2;\R)$-representation 
\[
\overline{\rho}:\pi_1 \overline{M}=\widetilde{\Gamma}(a_1,a_2,a_3)\rightarrow \SL(2;\R).
\]
\begin{lemma}
  The $\SL(2;\R)$-representation $\overline{\rho}$ is irreducible
  as an $\SL(2;\C)$-representation.
\end{lemma}
\begin{proof}
  Suppose that $\overline{\rho}$ were a reducible $\SL(2;\C)$-representation.
  We can move the image of $\pi_1 \overline{M}$ into
  the set of upper triangular matrices by conjugation.
  Under the homomorphism
  $\pi_1 \overline{M} \to \widetilde{\PSL}(2;\R)$, 
  we see that 
  $x^{a_1}$, $y^{a_2}$ and $z^{a_3}$ are sent to 
  $\sh(-\beta_1)$,
  $\sh(-\beta_2)$ and
  $\sh(-\beta_3)$ in $\widetilde{\PSL}(2;\R)$ respectively.
  The images of $x$, $y$ and $z$ are conjugate to
  $\sh(-\beta_1/a_1)$, $\sh(-\beta_2/a_2)$ and $\sh(-\beta_3/a_3)$
  in $\widetilde{\PSL}(2;\R)$.
  The projection from $\widetilde{\PSL}(2;\R)$ onto $\SL(2;\R)$
  sends $\sh(r)$ to a matrix of trace $2 \cos r \pi$.
  We can see that $\tr \overline{\rho}(xyz) \not = \pm 2$
  from that
  $\overline{\rho}(x)$, $\overline{\rho}(y)$,
  $\overline{\rho}(z)$ are upper triangular and that
  $a_1$, $a_2$, $a_3$ are pairwise coprime,
  $0 < \beta_i < a_i$ and $\sum_{i=1}^3 \beta_i / a_i < 1$.
  This is a contradiction to that $\overline{\rho}(xyz) = \overline{\rho}(h)^{\beta} = (-I)^{\beta}$.
\end{proof}

Note that for $\overline{M}=M(0;(1,\beta),(a_1,\beta_1),(a_2,\beta_2),(a_3,\beta_3))$, it holds that
\[
H_1(\overline{M};\Z)\cong\Z/a |e(\overline{M})|
\]
where $e(\overline{M})$ is the euler number of the Seifert fibration.

Since Brieskorn homology 3-sphere 
$\Sigma(a_1, a_2, a_3)$ is a Seifert fibered space
$M(0;(1,b), (a_1, b_1), (a_2, b_2), (a_3, b_3))$,
we can regard 
$\Sigma(a_1, a_2, a_3)$ 
as a universal abelian cover for each $\overline{M}$ with $a |e(\overline{M})| \not =1$.
We denote the induced homomorphism from $\pi_1 \Sigma(a_1, a_2, a_3)$ to 
$\pi_1 \overline{M}$ by $p$.
If an euler class $eu$ satisfies a condition 
in Theorem~\ref{thm:JankinsNeumann}, 
then we have an $\SL(2;\R)$-representation $\overline{\rho}$ of $\pi_1 \overline{M}$.
We also have an $\SL(2;\R)$-representation of $\pi_1 \Sigma(a_1, a_2, a_3)$
given by the pull--back $p^* \overline{\rho}$.

Hereafter we focus on $\SL(2;\R)$-representations $p^* \overline{\rho}$ 
which are induced from 
euler classes $eu$ satisfying that
$\beta=-1$ and $\beta_1 / a_1 + \beta_2 / a_2 + \beta_3 / a_3 < 1$.

\begin{lemma}
  We have a one--to--one correspondence from 
  euler classes satisfying the condition (a) 
  to 
  those satisfying (b) in Theorem~\ref{thm:JankinsNeumann}.
\end{lemma}
\begin{proof}
  This is given by an orientation reversing homeomorphism between Seifert fibered spaces. 
  Let $\overline{M}$ be 
  the Seifert fibered space corresponding to an euler class satisfying the condition~(a).
  By changing the Seifert invariant, we can express the Seifert invariant of $-\overline{M}$ as 
  \begin{align*}
    -\overline{M}
    &= -M(0;(1, -1), (a_1, \beta_1), (a_2, \beta_2), (a_3, \beta_3)) \\
    &=  M(0;(1, 1), (a_1, -\beta_1), (a_2, -\beta_2), (a_3, -\beta_3)) \\
    &=  M(0;(1, -2), (a_1, a_1-\beta_1), (a_2, a_2-\beta_2), (a_3, a_3-\beta_3)).
  \end{align*}
  This Seifert invariant satisfies the condition~(b).
\end{proof}

We have an $\SL(2;\R)$-representation of $\pi_1 \Sigma(a_1, a_2, a_3)$
by the pull-back induced from the universal abelian covering between
Seifert fibered spaces.
We write $\rho_{eu}$ for the pull--back of 
an $\SL(2;\R)$-representation of $\pi_1 \overline{M}$
corresponding to the euler class $eu$.
We denote by $\Phi$ the induced correspondence between
euler classes $eu$ and the conjugacy classes of $\rho_{eu}$, i.e., 
\[
\Phi \colon E \to \mathcal{R}_{\SL(2;\R)} (\subset \mathcal{R}_{\SL(2;\C)})
\]
where $E$ is the set of euler classes
$eu \in H^2(\Gamma(a_1,a_2,a_3);\Z)$
given by $eu = \beta x_0 + \beta_1 x_1 + \beta_2 x_2 + \beta_3 x_3$, 
$\beta =-1, 0 < \beta_i < a_i \,(i=1,2,3)$ and
$\sum_{i=1}^3 (\beta_i/a_i) < 1$.

\begin{remark}
  It is obvious that the set $E$ can be identified with $X_0$.
\end{remark}

\section{Main result}
In this section we prove the main theorem.
Here $M$ denotes a Brieskorn homology 3-sphere $\Sigma(a_1,a_2,a_3)$.
\begin{theorem}
  Let $\rho$ be an irreducible $\SL(2;\C)$-representation
  of $\pi_1 M$.
  If $\rho$ is not conjugate to an $\SU(2)$-representation, then
  there exits a unique euler class $eu \in E$ such that
  the induced $\SL(2;\R)$-representation $\rho_{eu}$ is conjugate to $\rho$.
\end{theorem}
\begin{proof}
  It is shown in~\cite{Boden-Curtis} that
  \[
  |\mathcal{R}_{\SL(2;\C)}| = \frac{(a_1-1)(a_2-1)(a_3-1)}{4}.
  \]
  Together with the result~\cite{Fintushel-Stern},
  we can see that
  $|\mathcal{R}_{\SL(2;\C)}| - |\mathcal{R}_{\SU(2)}| = |E|$.
  It remains to prove that
  \begin{enumerate}
  \item any irreducible $\SU(2)$-representation can not be moved to $\SL(2;\R)$-one
    under conjugation by a matrix in $\SL(2;\C)$,
  \item
    the map $\Phi$ is one--to--one.
  \end{enumerate}
  We divide these proofs into Proposition~\ref{prop:su_sl} and~\ref{prop:injectivity}.
\end{proof}

\begin{corollary}
  \label{remark:numbers_conj}
  We have that
  $|\mathcal{R}_{\SL(2;\C)}| = |\mathcal{R}_{\SU(2)}| + |\mathcal{R}_{\SL(2;\R)}|$
  by counting the conjugacy classes as $\SL(2;\C)$-representations.
\end{corollary}

\begin{remark}
  It remains the same to count the conjugacy classes of
  irreducible $\SU(2)$-representations 	 
  under the both conjugations by $\SU(2)$ and $\SL(2;\C)$.
  In the case of $\SL(2;\R)$, it is not same.
\end{remark}

\begin{remark}
  There exists an example in which has no irreducible $\SL(2;\R)$-representations.
  It holds for $\Sigma(2, 3, 5)$ that 
  $|\mathcal{R}_{\SL(2;\C)}| = |\mathcal{R}_{\SU(2)}|$.
\end{remark}

The following proposition is folklore. 
\begin{proposition}
  \label{prop:su_sl}
  Let $\rho:\pi_1\Sigma(a_1,a_2,a_3) \rightarrow \SU(2)$ be an irreducible representation. 
  If $\rho$ is conjugate to an $\SL(2;\R)$-representation by a matrix in $\SL(2;\C)$,
  then $\rho$ is abelian.
\end{proposition}

Proposition~\ref{prop:su_sl} follows from the next lemma.
\begin{lemma}
  Let $A, B \in \SU(2)$. 
  If there exists $P\in\SL(2;\C)$ such that $PAP^{-1}, PBP^{-1} \in \SL(2;\R)$, 
  then $AB=BA$. 
\end{lemma}
\begin{proof}
  We prove that  if 
  $P^{-1}AP$ and $P^{-1}BP$ are matrices in $\SL(2;\R)$, then
  $AB=BA$. It means that $A$ and $B$ are contained in the same maximal abelian subgroup in $\SU(2)$. 
  It is known that every matrix in $\SU(2)$ can be diagonalizable by a matrix in $\SU(2)$.
  We can assume $A$ is a diagonal matrix.
  Suppose that $P^{-1}AP, P^{-1}BP \in \SL(2;\R)$.
  We write $T_A$ to the linear fraction transform by $A$ and so on.
  It is known that $T_{P^{-1}AP}$ has the fixed points given by $a$ and $\bar a$ (the complex conjugate of $a$).
  The fixed points of $T_A$ are $0$ and $\infty$.
  Since $T_P$ sends $a$ and $\bar a$ to $0$ and $\infty$ respectively,
  we can express $T_P$ as
  \[T_P = e^{i\theta}\, \frac{z-a}{z-\bar a} \quad (\theta \in \R).\]
  It is known that this $T_P$ maps the real line to the unit circle in the complex plane.
  Together with that the real linear fractional transform $T_{P^{-1}BP}$ maps the real line onto itself,
  we can see that $T_{B} = T_{P} \circ T_{P^{-1}BP} \circ T_{P}^{-1}$ maps the unit circle onto itself.
  Hence $T_B$ is expressed as
  \[T_B = e^{i\varphi}\, \frac{z-\zeta}{-\bar{\zeta} z +1} \quad (\varphi \in \R).\]
  On the other hand, we can express $T_B$ as
  \[T_B = \frac{\xi z + \eta}{-\bar{\eta} z + \bar{\xi}}, \quad
    |\xi|^2 + |\eta|^2 =1 \]
  since $B$ is an element in $\SU(2)$.
  Since $\eta$ and $\zeta$ must be zero, $B$ is also diagonal.
  This means that $A$ and $B$ are contained in the same maximal abelian subgroup in $\SU(2)$.
\end{proof}
\begin{remark}
  We have shown a more general situation.
  Let $F$ be a finitely generated group 
  and $\rho:F \rightarrow \SU(2)$ a representation. 
  If $\rho$ is conjugate to an $\SL(2;\R)$-representation by a matrix in $\SL(2;\C)$, 
  then $\rho$ is abelian. 
\end{remark}

Under an irreducible representation
$\rho:\pi_1\Sigma(a_1,a_2,a_3)\rightarrow \SL(2;\C)$,
we write the capital letter $X$ for $\rho(x)$.
Any conjugacy class in $\mathcal{R}_{\SL(2;\C)}$ 
can be parametrized by the triple of traces 
$(\tr X, \tr Y, \tr XY)$. 
This triple is determined under a choice of integers $b$ and $b_i$ in the Seifert invariant.
\begin{remark}
  These integers $b,b_1,b_2,b_3$ are not unique for given $a_1,a_2,a_3$. 
  As in~\cite{Boden-Curtis}, we can choose $b=0$.
  We will require this assumption in the proof of our main result.
\end{remark}

We may assume $b_2$ is even because $a_2$ is odd. 
If $b_2$ is odd, 
then we can replace 
$b_1$ and $b_2$
with 
$b_1+a_1$ and $b_2-a_2$ respectively. 
We will show that $\Phi$ is a one--to--one correspondence between the euler class
and the triple of traces. 
First we assume that integers $b$ and $b_i$ in the Seifert invariant 
of $\Sigma(a_1, a_2, a_3)$ satisfy that 
$b=0$ and $b_2$, $b_3$ are even.
Under this assumption, 
the presentation of $\pi_1 \Sigma(a_1, a_2, a_3)$ turns into 
\[
\langle
x,y,z,h\ |\ h\text{ central}, x^{a_1}=h^{-b_1},y^{a_2}=h^{-b_2},z^{a_3}=h^{-b_3},xyz=1
\rangle.
\]
The trace $\tr Z$ equals to $\tr (XY)^{-1} = \tr XY$.
The triple $(\tr X,\tr Y, \tr XY)$ coincides with
$(\tr X,\tr Y, \tr Z)$.

The following two lemmas give a necessary condition for  
$(\tr X,\tr Y, \tr Z)$ to coincide with $(\tr X',\tr Y', \tr Z')$.
We set integers $q_i$ and $r_i$ as 
\[-a|e(\overline{M})|b_i = a_i q_i + r_i\quad (0 < r_i <a_i).\]
\begin{lemma}
  \label{lemma:traces_rho_e}
  The trace of $X$ under $\rho_{eu}$ is expressed as
  \[
  \tr X = (-1)^{q_1} 2 \cos \frac{\beta_1}{a_1} \pi
  \]
  Similarly, the traces of $Y$ and $Z$ are also expressed as
  \[
  \tr Y = (-1)^{q_2} 2 \cos \frac{\beta_2}{a_2} \pi,
  \quad
    \tr Z = (-1)^{q_3} 2 \cos \frac{\beta_3}{a_3} \pi.
  \]
\end{lemma}
\begin{proof}
  By definition, the $\SL(2;\R)$-representation $\rho_{eu} = p^* \overline{\rho}$
  factors through
  $\pi_1 \Sigma(a_1, a_2, a_3) \to \pi_1 \overline{M} \to \widetilde{\PSL}(2;\R).$
  We denote the images in $\widetilde{\PSL}(2;\R)$ of
  the generators $x$, $y$, $z \in \pi_1 \Sigma(a_1, a_2, a_3)$
  by $\tilde{x}$, $\tilde{y}$, $\tilde{z}$ respectively.
  Since the order $|H(\overline{M};\Z)|$ is $a|e(\overline{M})|$, the covering degree
  between $\Sigma(a_1, a_2, a_3)$ and $\overline{M}$ also equals to $a|e(\overline{M})|$.
  The regular fiber $h \in \pi_1 \Sigma(a_1, a_2, a_3)$ is sent to
  $\sh(a|e(\overline{M})|)$ in $\widetilde{\PSL}(2;\R)$.
  By the relation $x^{a_1} = h^{-b_1}$, we have the relation
  $\tilde{x}^{a_1} = \sh(-a|e(\overline{M})|b_1)$ in $\widetilde{\PSL}(2;\R)$.
  Hence $\tilde{x}$ is conjugate to $\sh(-a|e(\overline{M})|b_1 / a_1)$.

  Since 
  the image of $\sh(r)$ has trace $2 \cos r \pi$
  under the projection onto $\SL(2;\R)$,
  we have that
  \begin{equation}
    \label{eqn:trace_X}\tag{$\ast$}
    \tr X
    = 2 \cos \frac{-a|e(\overline{M})|b_1}{a_1} \pi
    = (-1)^{q_1} 2 \cos \frac{r_1}{a_1} \pi
  \end{equation}
  where $-a|e(\overline{M})|b_1 = a_1 q_1 + r_1\, (0 < r_1 < a_1)$.
  
  Finally we recall the relation that $a(b + \sum_{i=1}^3 b_i/a_i) = 1$.
  This relation implies that 
  $a_2 a_3 b_1 \equiv 1 \pmod{a_1}$.
  Therefore we have that
  $-a|e(\overline{M})|b_1 \equiv \beta_1a_2a_3b_1 \equiv \beta_1 \pmod{a_1}$.
  Together with \eqref{eqn:trace_X}, we obtain our claim.  
\end{proof}

\begin{lemma}
  \label{lemma:necessary_condition}
  If $\rho_{eu}$ is conjugate to $\rho_{eu'}$, then
  $\beta_i = \beta'_i$ for all $i$ or
  there is a unique $i$ such that $\beta'_i = a_i - \beta_i$ and
  $\beta_j = \beta'_j$ for $j \not = i$.
\end{lemma}
\begin{proof}
  By Lemma~\ref{lemma:traces_rho_e}, the equality $\tr X = \tr X'$,
  $\tr Y = \tr Y'$ and $\tr Z = \tr Z'$
  give the constraints that 
  $\beta_i = \beta'_i$ or $\beta'_i = a_i - \beta_i$ for each $i$.
  From another constraints that
  $\beta_1 / a_1 + \beta_2 / a_2 + \beta_3 / a_3 < 1$ and 
  $\beta'_1 / a_1 + \beta'_2 / a_2 + \beta'_3 / a_3 < 1$,
  we must have only one different $\beta'_i$ from $\beta_i$ if there exists.
\end{proof}

\begin{remark}
  \label{remark:necessary_condition}
  The conclusions in Lemma~\ref{lemma:necessary_condition} are not exclusive.
  When $a_1$ equals to $2\beta_1$, the case of $i=1$ in the latter conclusion
  coincides with the former one. 
\end{remark}

\begin{proposition}
  \label{prop:injectivity}
  The map $\Phi: E \to \mathcal{R}_{\SL(2;\R)}$ is one--to--one.
\end{proposition}
\begin{proof}
  We prove that if the triple $(\tr X,\tr Y, \tr Z)$ 
  equals to $(\tr X',\tr Y', \tr Z')$, then $e = e'$.
  From Lemma~\ref{lemma:necessary_condition} and
  Remark~\ref{remark:necessary_condition}, it remains to show that 
  if there exists a unique $i$ such that $\beta'_i = a_i - \beta_i$,
  then $a_i = 2\beta_i$. Consequently, $e$ coincides with $e'$.
  We will show that $a_1 = 2\beta_1$ under our assumption that
  $b=0$, $a_2$, $a_3$ are odd and $b_2$, $b_3$ are even.
  Suppose that $\beta'_1 = \beta_1 - a_1$.
  By the equation~\eqref{eqn:trace_X}, we can express $\tr X$ as
  \begin{align*}
    \tr X
    &= 2 \cos \frac{-a|e(\overline{M})|b_1}{a_1} \pi\\
    &= 2 \cos \left(-b_1a_2a_3 + \frac{\beta_1 b_1 a_2 a_3}{a_1}
    + b_1 \beta_2 a_3 + b_1 a_2 \beta_3 \right) \pi \\
    &= -2 \cos \left(\frac{\beta_1 b_1 a_2 a_3}{a_1}
    + b_1 \beta_2 a_3 + b_1 a_2 \beta_3 \right) \pi,
  \end{align*}
  which is due to that $b_1 a_2 a_3$ is odd.
  On the other hands, we can also express $\tr X'$ as
  \begin{align*}
    \tr X'
    &= 2 \cos \left( -\frac{\beta_1 b_1 a_2 a_3}{a_1}
    + b_1 \beta_2 a_3+ b_1a_2\beta_3\right)\pi \\
    &= 2 \cos \left( -\frac{\beta_1 b_1 a_2 a_3}{a_1}
    - b_1 \beta_2 a_3 - b_1a_2\beta_3\right)\pi.
  \end{align*}
  Since we assume that $\tr X = \tr X'$,
  we have $\tr X = 0$ which shows that $a_1 = 2\beta_1$
  from Lemma~\ref{lemma:traces_rho_e}.
  The same conclusion $a_1 = 2\beta_1$ can be drawn for 
  the cases that $\beta_i' = a_i - \beta_i$ for $i=2, 3$
  by similar arguments.
  These cases are excluded by 
    the constraints that $\sum_i \beta_i / a_i < 1$ and $\sum_i \beta'_i / a_i < 1$.
\end{proof}

We can use the set of euler classes satisfying the condition (b) in
Theorem~\ref{thm:JankinsNeumann}.
This is due to that 
the Seifert fibered space $-\overline{M}$ also induces the same triple of traces
as that for $\overline{M}$. 
We can express an isomorphism
$\varphi:\pi_1 \overline{M} \to \pi_1 (-\overline{M})$
as
\[\varphi: h \mapsto h'^{-1}, x \mapsto x'h', y \mapsto y'h', z \mapsto z'h'.\]
Let $\rho$ and $\rho'$ be
the induced $\SL(2;\R)$-representations by $\overline{M}$ and $-\overline{M}$
respectively. 
The triple for $\rho$ is equal to that of $\rho'$.
Note that $e(-\overline{M}) = -e(\overline{M})$ and
$h$ is send to $\sh(-a|e(-\overline{M})|)$
in $\widetilde{\PSL}(2;\R)$ under $\rho'$.
By a similar argument in the proof of Lemma~\ref{lemma:traces_rho_e},
we can show that 
\[
  \tr \rho' (x)
  = 2 \cos \frac{a|e(-\overline{M})|b_1}{a_1}\pi 
  = (-1)^{-q_1} 2 \cos \left(\frac{-\beta_1}{a_1} \pi \right) 
  = \tr \rho(x).
\]
Also we have $\tr \rho' (y) = \tr \rho(y)$ and $\tr \rho' (z) = \tr \rho(z)$.


\section{Examples}
\subsection{The Brieskorn homology 3-sphere $\Sigma(2,3,6n+1)$}
We consider the $\SL(2;\R)$-representations of 
$\pi_1 \Sigma(2,3,6n+1)$.
It is known that the Brieskorn homology 3-sphere $\Sigma(2,3,6n+1)$
is obtained by $1/n$-surgery along the torus knot of type $(2, 3)$ in $S^3$.
Here we assume that $n > 0$.
The Casson invariant $|\lambda(\Sigma(2,3,6n+1))|$ equals to $n$ and 
the $\SL(2;\C)$-Casson invariant $\lambda_{\SL(2;\C)}(\Sigma(2,3,6n+1))$ equals to $3n$ 
by \cite[Theorem~2.3]{Boden-Curtis}.
Since the difference $\lambda_{\SL(2;\C)}(\Sigma(2,3,6n+1)) - 2|\lambda(\Sigma(2,3,6n+1))|$
is $n$,
we have $n$ conjugacy classes of irreducible $\SL(2;\R)$-representations.

\subsubsection*{The euler classes of $\Gamma(2, 3, 6n+1)$}
There exist $n$ euler classes 
$\beta x_0 + \beta_1 x_1 + \beta_2 x_2 + \beta_3 x_3 \in H^2(\Gamma(2,3,6n+1))$
such that
$\beta = -1$ and $\beta_1 / 2 + \beta_2 / 3 + \beta_3 / (6n+1) < 1$.
They are given by
\[-x_0 + x_1 + x_2 + k x_3\]
where $1 \leq k \leq n$.

We denote the corresponding Seifert fibered space by $\overline{M}_k$.
The order of $H_1(\overline{M}_k;\Z)$ is given by $6(n-k)+1$.
Note that $\overline{M}_n$ is the Brieskorn homology 3-sphere $\Sigma(2, 3, 6n+1)$.
We can change the Seifert invariant of $\overline{M}_n$ from 
$\{0; (1,-1), (2,1), (3,1), (6n+1,n)\}$ to $\{0;(1,0),(2,1),(3,-2),(6n+1,n)\}$.
The fundamental group of $\pi_1 \overline{M}_n$ is expressed as
\[
\pi_1 \overline{M}_n
=\langle x, y, z, h \,|\, \hbox{$h$ central}, x^2=h^{-1}, y^3=h^2, z^{6n+1}=h^{-n}, xyz=1\rangle
\]

\subsubsection*{The triples of $\tr X$, $\tr Y$ and $\tr Z$}
By Lemma~\ref{lemma:traces_rho_e}, we have the $\SL(2;\R)$-representations
induced from the euler classes. The corresponding triples of traces are given by 
\begin{align*}
  \tr X &= 2 \cos \left(-\frac{6(n-k)+1}{2} \pi \right) =0, \\
  \tr Y &= 2 \cos \left(\frac{2(6(n-k)+1)}{3} \pi \right) = -1, \\
  \tr Z &= 2 \cos \left(-\frac{(6(n-k)+1)n}{6n+1} \pi \right)  \\
  &=  2 \cos \left(- \Big(n-k + \frac{k}{6n+1}\Big) \pi \right) \\
  &= \begin{cases}
    2 \cos \left(\frac{k}{6n+1} \pi \right) & \hbox{$n-k$ is even}; \\
    2 \cos \left(\pi - \frac{k}{6n+1}\pi \right) & \hbox{$n-k$ is odd}.
  \end{cases}
\end{align*}
Note that the trace $\tr Z = 2\cos \frac{\ell}{6n+1}\pi$ for an $\SU(2)$-representation
must satisfy that $\ell \equiv n \pmod{2}$ and 
\[
\frac{1}{6} < \frac{\ell}{6n+1} < \frac{5}{6}.
\]
We can express explicit $\tr Z$ for $\SU(2)$-representations as
\[
\tr Z =
2 \cos \left(\frac{n+2}{6n+1} \pi \right),
2 \cos \left(\frac{n+4}{6n+1} \pi \right),
\ldots,
2 \cos \left(\frac{5n}{6n+1} \pi \right).
\]

We have seen that $\Sigma(2, 3, 6n+1)$ appears in $\overline{M}_k$.
There exist examples in which $\Sigma(a_1, a_2, a_3)$ does not appear in the sequence of
Seifert fibered spaces induced by euler classes in $E$.

\subsection{The Brieskorn homology 3-sphere $\Sigma(3,5,7)$}
Under the constraint that $b=-1$ 
and $0 < b_1 < 3$, $0 < b_2 < 5$ and $0 < b_3 < 7$,
we can choose the Seifert invariant as
$\{0; (1,-1), (3, 2), (5, 1), (7, 1)\}$.
Since $2/3 + 1/5 + 1/7 >1$, 
$\Sigma(3,5,7)$ does appear in the sequence of $\overline{M}$.
In fact, we have $4$ euler classes satisfying 
that the condition (a) in Theorem~\ref{thm:JankinsNeumann}.
They are given by
\begin{gather*}
  -x_0 + x_1 + x_2 + x_3,  \quad -x_0 + x_1 + x_2 + 2x_3, \\
  -x_0 + x_1 + x_2 + 3x_3, \quad -x_0 + x_1 + 2x_2 + x_3.
\end{gather*}
We can see that
every corresponding Seifert fibered space $\overline{M}$ has the non--trivial
homology group $H_1(\overline{M};\Z)$
by calculating the orders.

\bigskip

\noindent
\textit{Acknowledgments}.
This paper was written while the authors were visiting 
Aix-Marseille University. 
They would like to express their sincere thanks for the hospitality. 
They also wish to express their thanks to Joan Porti for his helpful comments.
This research was supported in part by JSPS KAKENHI  25400101 and 26800030. 


\end{document}